\documentclass[11pt,leqno]{article}
\usepackage{amsmath, amscd, amsthm, amssymb, graphics, xypic, mathrsfs, setspace, fancyhdr, times, bm, enumitem}
\DeclareFontFamily{OT1}{pzc}{}
\DeclareFontShape{OT1}{pzc}{m}{it}%
             {<-> s * [1.195] pzcmi7t}{}
\DeclareMathAlphabet{\mathscr}{OT1}{pzc}%
                                 {m}{it}
\usepackage[colorlinks=true,pagebackref=true]{hyperref} 
\hypersetup{backref}

\newcommand{\Spec}{\operatorname{Spec}}

\newcommand{\sma}{{\scriptstyle{\wedge}}}

\renewcommand{\O}{{\mathcal O}}

\newcommand{\colim}{\operatorname{colim}}

\newcommand{\cplx}{{\mathbb C}}

\newcommand{\Z}{{\mathbb Z}}

\newcommand{\aone}{{\mathbb A}^1}
\newcommand{\pone}{{\mathbb P}^1}

\newcommand{\gm}[1]{{{\mathbf G}_{m}^{#1}}}

\newcommand{\ho}[1]{\mathscr{H}({#1})}
\newcommand{\hop}[1]{\mathscr{H}_{\bullet}({#1})}

\newcommand{\bpi}{\bm{\pi}}

\newcommand{\Sm}{\mathrm{Sm}}

\newcommand{\Spc}{\mathrm{Spc}}

\newcommand{\K}{{{\mathbf K}}}

\newcommand{\Addresses}{{
 \bigskip
 \footnotesize

 A.~Asok, Department of Mathematics, University of Southern California, 3620 S. Vermont Ave.,
  Los Angeles, CA 90089-2532, United States; \textit{E-mail address:} \url{asok@usc.edu}

  \medskip

 J.~Fasel, Institut Fourier - UMR 5582, Universit\'e Grenoble-Alpes, CS 40700, 38058 Grenoble Cedex 9; France \textit{E-mail address:} \url{jean.fasel@gmail.com}
}}

\newcounter{intro}
\theoremstyle{plain}
\newtheorem{thm}{Theorem}[subsection]

\newtheorem{lem}[thm]{Lemma}
\newtheorem{cor}[thm]{Corollary}
\newtheorem{prop}[thm]{Proposition}
\newtheorem*{claim*}{Claim}  

\newtheorem*{thm*}{Theorem}
\newtheorem*{problem*}{Problem}

\newtheorem{thmintro}{Theorem}

\newtheorem{questionintro}[thmintro]{Question}

\theoremstyle{definition}

\theoremstyle{remark}
\newtheorem{rem}[thm]{Remark}
\newtheorem{remintro}[thmintro]{Remark}

\numberwithin{equation}{section}

\begin{document}
\pagestyle{fancy}
\renewcommand{\sectionmark}[1]{\markright{\thesection\ #1}}
\fancyhead{}
\fancyhead[LO,R]{\bfseries\footnotesize\thepage}
\fancyhead[LE]{\bfseries\footnotesize\rightmark}
\fancyhead[RO]{\bfseries\footnotesize\rightmark}
\chead[]{}
\cfoot[]{}
\setlength{\headheight}{1cm}

\author{Aravind Asok\thanks{Aravind Asok was partially supported by National Science Foundation Awards DMS-0966589 and DMS-1254892.} \and Jean Fasel\thanks{Jean Fasel was supported by the DFG Grant SFB Transregio 45.}}

\title{{\bf Algebraic vs. topological vector bundles on spheres}}
\date{}
\maketitle

\begin{abstract}
We study the problem of when a topological vector bundle on a smooth complex affine variety admits an algebraic structure.  We prove that all rank $2$ topological complex vector bundles on smooth affine quadrics of dimension $11$ over the complex numbers admit algebraic structures.
\end{abstract}


\section{Introduction}
If $X$ is a smooth complex algebraic variety, write $\mathscr{V}_n(X)$ for the set of isomorphism classes of rank $n$ algebraic vector bundles on $X$, and $\mathscr{V}_n^{top}(X)$ for the set of isomorphism classes of rank $n$ complex topological vector bundles on $X^{an}:=X(\cplx)$ viewed as a complex manifold.  There is a function ``forget the algebraic structure"
\[
{\mathfrak R}_{n,X}: \mathscr{V}_n(X) \longrightarrow \mathscr{V}_n^{top}(X).
\]
A complex topological vector bundle of rank $n$ on $X^{an}$ is called algebraizable if it lies in the image of $\mathfrak{R}_{n,X}$.

It is a classical (and difficult) problem to construct indecomposable vector bundles of low rank on ``simple" smooth algebraic varieties.  Rather than attempting to make the notion of ``simple" precise, we focus on an example.  Let $Q_{2n-1}$ be the smooth quadric hypersurface in ${\mathbb A}^{2n}$ cut out by the equation $\sum_{i=1}^n x_i y_i = 1$.  The underlying complex manifold $Q_{2n-1}^{an}$ is homotopy equivalent to the sphere $S^{2n-1}$.  The goal of this note is to establish the following result.

\begin{thmintro}
\label{thmintro:main}
The map ${\mathfrak R}_{2,Q_{11}}$ is surjective, i.e., every rank $2$ topological vector bundle on $Q_{11}$ admits an algebraic structure.  In particular, there exist indecomposable rank $2$ vector bundles on $Q_{11}$.
\end{thmintro}

\begin{remintro}
With some additional analysis, it is possible to prove that ${\mathfrak R}_{r,Q_{2n-1}}$ is actually surjective for arbitrary $r$ and $n \leq 6$. ``Stable" versions of results like Theorem \ref{thmintro:main} are classical \cite{Fossum}.  There are also some unstable results for spheres of small dimension \cite{Moore}.  From a modern point of view, all of these results are easy to establish; see for example the introduction of \cite{AsokFaselHopkins} for a discussion of the algebraizability problem for arbitrary smooth varieties of small dimension.
\end{remintro}

\begin{remintro}
For a general smooth complex affine $X$ the map ${\mathfrak R}_{r,X}$ need not be either injective or surjective:  injectivity can fail for $r = 1$ on smooth complex affine curves, and surjectivity can fail for $r = 1$ on smooth complex affine surfaces.  Indeed, by Grauert's Oka-principle \cite{Grauert}, the classifications of analytic and topological vector bundles on a Stein manifold, e.g., a smooth complex affine variety, coincide.  On the other hand, it follows from the long exact sequence in cohomology attached to the exponential sheaf sequence that smooth affine curves of positive genus have no non-trivial analytic line bundles, even though they have plenty of non-trivial algebraic line bundles.
\end{remintro}

\begin{remintro}
One reason why the case $r = 2$ and $n = 6$ is especially interesting is as follows.  It is an open problem to determine whether, when $X = {\mathbb P}^n$, $n \geq 4$, all rank $2$ vector bundles are algebraizable (see, e.g., \cite[Chapter 1 \S 6.5]{OSS}).  There is a smooth surjective morphism $Q_{2n-1} \to {\mathbb P}^{n-1}$.  While we do not know, e.g., whether all rank $2$ vector bundles on ${\mathbb P}^5$ are algebraizable, Theorem \ref{thmintro:main} implies that any rank $2$ topological complex vector bundle on ${\mathbb P}^5$ becomes algebraic after pullback to $Q_{11}$.  The problem of whether the original bundle on ${\mathbb P}^5$ is algebraizable can then be viewed as a question in descent theory.
\end{remintro}

The set $\mathscr{V}_r^{top}(Q_{2n-1})$ is, by means of the homotopy equivalence $Q_{2n-1}^{an} \cong S^{2n-1}$, in bijection with the set of free homotopy classes of maps $[S^{2n-1},BU(r)]$.  Because $BU(r)$ is simply connected, the canonical map from pointed to free homotopy classes of maps is a bijection, i.e., $\pi_{2n-1}(BU(r)) \to [S^{2n-1},BU(r)]$ is a bijection.  On the other hand, for $n \geq 3$, the map $\pi_{2n-1}(BSU(r)) \to \pi_{2n-1}(BU(r))$ is an isomorphism.  In the special case where $r = 2$ we know that $\pi_{2n-1}(BSU(2)) \cong \pi_{2n-2}(SU(2)) \cong \pi_{2n-2}(S^3)$.

Fix a field $k$ and write $\ho{k}$ for the Morel-Voevodsky $\aone$-homotopy category of $k$-schemes.  F. Morel gave an algebraic analog of Steenrod's celebrated homotopy classification of vector bundles: there is an $\aone$-homotopy classification of algebraic vector bundles on smooth affine schemes; see \cite[Theorem 1]{AHW} for a precise statement.  As explained in the introduction to \cite{AsokFaselSpheres}, by a procedure analogous to that described in the previous paragraph, there is a canonical bijection
\[
\mathscr{V}^o_r(Q_{2n-1}) \cong \bpi_{n-1,n+1}^{\aone}(SL_r),
\]
where $\mathscr{V}^o_r(X)$ is the set of isomorphism classes of oriented vector bundles on a smooth scheme $X$ (see also \cite[Theorem 4.1.1]{AHWII}), i.e., vector bundles with a chosen trivialization of the determinant.

The comparison between the results of the previous two paragraphs is facilitated by ``complex realization," which provides a homomorphism
$\bpi_{n-1,n+1}^{\aone}(SL_2)(\cplx) \to \pi_{2n}(S^3)$.  To establish the result above, it suffices to prove the displayed homomorphism is surjective.  Thus, the above algebraizability question boils down to a question regarding $\aone$-homotopy sheaves of $SL_2$.  Theorem \ref{thmintro:main} is then a consequence of the following result.

\begin{thmintro}[See Theorem \ref{thm:oddtorsion}]
\label{thmintro:maincomputation}
The homomorphism
\[
\bpi_{4,6}^{\aone}(SL_2)(\cplx) \longrightarrow \pi_{10}(S^3) \cong \Z/15.
\]
is surjective.
\end{thmintro}

The above result suggests the following question (which has a positive answer for primes $p \leq 5$).

\begin{questionintro}
If $p$ is a prime number, is the homomorphism
\[
\bpi_{p-1,p+1}^{\aone}(SL_2)(\cplx) \longrightarrow \bpi_{2p}(S^3)
\]
surjective on $p$-components?
\end{questionintro}

That $\pi_{10}(S^3) \cong \Z/15$ is classical (see, e.g., \cite{Toda}), so to establish Theorem \ref{thmintro:maincomputation} it suffices simply to produce a lift of a generator.  The group $\pi_{10}(S^3)$ is especially interesting because it is the first place where $5$-torsion appears in homotopy groups of $S^3$.  To prove the main result, we introduce a spectral sequence whose $E_1$-page involves $\aone$-homotopy sheaves of punctured affine spaces and that converges to the $\aone$-homotopy sheaves of the stable symplectic group; this is achieved in Subsection \ref{ss:symplecticss}.  In Section \ref{s:oddtorsion}, we analyze the ``symplectic spectral sequence" constructed above to produce the lift of a generator of $\pi_{10}(S^3)$.  Finally, Section \ref{s:explicit} begins to study the problem of explicitly constructing the rank $2$ algebraic vector bundle whose existence is guaranteed by Theorem \ref{thmintro:main}.  We show that any rank $2$ algebraic vector bundle whose associated classifying map corresponds to an element of $\bpi_{4,6}^{\aone}(SL_2)(\cplx)$ not lying in the kernel of the map in Theorem \ref{thmintro:maincomputation}, remains non-trivial after adding a trivial bundle of rank $\leq 3$, but becomes trivial after adding a trivial summand of rank $4$.

\subsubsection*{Notation/Preliminaries}
Throughout we fix a base-field $k$.  This note uses much notation from \cite{AsokFaselSpheres,AsokFaselThreefolds} and \cite{AsokFaselpi3a3minus0}, and our conventions and notation will follow those papers.  We write $\Sm_k$ for the category of schemes that are separated, finite type and smooth over $\Spec k$, and $\Spc_k$ for the category of simplicial presheaves on $\Sm_k$ (objects of this category will be called $k$-{\em spaces}).  As usual, given $X \in \Sm_k$ we view $X$ as a simplicial presheaf by considering the simplicially constant object associated with representable presheaf on $\Sm_k$ defined by $X$.  We write $\hop{k}$ for the Morel-Voevodsky pointed $\aone$-homotopy category; this category is obtained as a Bousfield localization of $\Spc_k$.

We write $S^i$ for the constant presheaf defined by the usual simplicial $i$-sphere.  It will be useful to remember that the quadric $Q_{2n-1} \subset {\mathbb A}^{2n}$ defined by $\sum_i x_iy_i = 1$ is isomorphic in $\hop{k}$ to ${\mathbb A}^n \setminus 0$ by projection onto the $x$-variables. It follows, for example, that the map $SL_2 \to {\mathbb A}^2 \setminus 0$ corresponding to projection onto the first column is an $\aone$-weak equivalence.   Moreover, ${\mathbb A}^n \setminus 0$ is $\aone$-weakly equivalent to $S^{n-1} \sma \gm{\sma n}$.  Also, $\pone \cong S^1 \sma \gm{}$, and thus $({\mathbb A}^{n-1}\setminus 0) \sma \pone \cong {\mathbb A}^{n}\setminus 0$ as well.

If $(\mathcal{X},x)$ is a pointed simplicial Nisnevich sheaf on $\Sm_k$, we define $\bpi_{i,j}^{\aone}(\mathcal{X},x)$ as the Nisnevich sheaf on $\Sm_k$ associated with the presheaf
\[
U \longmapsto \operatorname{Hom}_{\hop{k}}(S^i_s \sma \gm{\sma j} \sma U_+,(\mathcal{X},x)).
\]
We write $\K^Q_i$ for the Nisnevich sheafification of the Quillen K-theory presheaves $U \mapsto K_i(U)$, and $\mathbf{GW}^j_i$ for the Nisnevich sheafification of the Grothendieck-Witt groups $GW^j_i(U,\O_U)$, $\K^M_i$ for the unramified Milnor K-theory sheaves and $\K^{MW}_i$ for the unramified Milnor--Witt K-theory sheaves (see \cite[Chapter 3]{MField} for a detailed discussion of the latter notions).  We freely use the identification $\bpi_{n-1,j}^{\aone}({\mathbb A}^n \setminus 0) \cong \K^{MW}_{n-j}$ \cite[Corollary 6.43]{MField}.

If $\iota: k \hookrightarrow \cplx$ is a fixed embedding, then there is an induced ``complex realization" functor ${\mathfrak R}_{\iota}: \hop{k} \to \mathscr{H}$, where $\mathscr{H}$ is the usual homotopy category.  For a more detailed discussion of this construction, we refer the reader to \cite[\S 3.3.2]{MV} or \cite{DuggerIsaksen}.

\subsubsection*{Acknowledgements}
The first named author would like to thank B. Doran for discussions about quadrics and M. Hopkins for discussions about the problem of algebraizing vector bundles on ${\mathbb P}^n$. We sincereley thank the referee for useful comments.  This paper was initially posted in early 2014, but since then we have made progress on some of the questions above in joint work with M. Hopkins; we refer the reader to \cite{AsokFaselHopkins,AsokFaselHopkinsII} and the references therein for more details.

\section{Some spectral sequences}
\label{s:spectralsequence}
The goal of this section is to describe some spectral sequences whose $E_1$-pages are homotopy sheaves of punctured affine spaces and that converge to algebraic K-theory sheaves or Grothendieck-Witt sheaves.  These spectral sequences are the algebro-geometric cousins of the ``orthogonal spectral sequence" studied by Mahowald (see the discussion just subsequent to \cite[Diagram 1.5.14]{Ravenel}).

\subsection{The linear spectral sequence}
\label{ss:linearss}
For any integer $n \geq 2$ consider the inclusion $GL_{n-1} \hookrightarrow GL_{n}$ sending a matrix $X$ to the block-matrix $diag(X,1)$.  We write $GL = \colim_n GL_n$ with respect to these inclusions.  The quotients $GL_{n}/GL_{n-1}$ are isomorphic to $Q_{2n-1}$ by the map sending an $n \times n$-matrix to its first row and the first column of its inverse.  Since $Q_{2n-1}$ is $\aone$-weakly equivalent to ${\mathbb A}^n \setminus 0$, one deduces that there are $\aone$-fiber sequences of the form
\[
GL_{n-1} \longrightarrow GL_n \longrightarrow {\mathbb A}^n \setminus 0.
\]
For any integer $j \geq 0$, these $\aone$-fiber sequences induce long exact sequences in $\aone$-homotopy sheaves of the form
\[
\cdots \longrightarrow \bpi_{i,j}^{\aone}(GL_{n-1}) \longrightarrow \bpi_{i,j}^{\aone}(GL_n) \longrightarrow \bpi_{i,j}^{\aone}({\mathbb A}^{n} \setminus 0) \longrightarrow \bpi_{i-1,j}^{\aone}(GL_{n-1}) \longrightarrow \cdots.
\]
Putting these sequences together yields an exact couple of the form
\[
\xymatrix{
\bigoplus_{n \geq 1,i\geq0} \bpi_{i,j}^{\aone} (GL_{n}) \ar[rr]& & \bigoplus_{n \geq 1,i\geq0}\bpi_{i,j}^{\aone}(GL_{n}) \ar[dl] \\ & \bigoplus_{n \geq 1,i\geq0} \bpi_{i,j}^{\aone}({\mathbb A}^{n} \setminus 0) \ar[ul] &.
}
\]
Since taking homotopy groups commutes with filtered colimits (this is classical for simplicial sets), we conclude that $\colim_n \bpi_{i,j}^{\aone}(GL_n) = \bpi_{i,j}^{\aone}(GL)$.  Using this fact, we summarize the structure of the associated spectral sequence in the next result.

\begin{prop}
\label{prop:linearss}
For any integer $j \geq 0$, there is a spectral sequence with
\[
E^1_{p,q}(j) := \bpi_{p+q,j}^{\aone}({\mathbb A}^{p} \setminus 0)
\]
and converging to $\bpi_{p+q,j}^{\aone}(GL) = \K^Q_{p+q-j+1}$.
\end{prop}

\begin{proof}
Since we observed $\colim_{n} \bpi_{i,j}^{\aone}(GL_n) = \bpi_{i,j}^{\aone}(GL)$, and since $\bpi_{i,j}^{\aone}({\mathbb A}^p \setminus 0)$ vanishes for $i \leq p-1$, the convergence statement follows immediately from the classical convergence statement for exact couples \cite[Theorem 5.9.7]{Weibel}.  The identification of the homotopy sheaves of the stable general linear group with Quillen K-theory sheaves follows from $\aone$-representability of algebraic K-theory \cite[\S 4 Theorem 3.13]{MV}.
\end{proof}

\subsection{The symplectic spectral sequence}
\label{ss:symplecticss}
We now analyze a variant of the above spectral sequence replacing the general linear group by the symplectic group.  As before, there are standard ``stabilization" embeddings $Sp_{2n-2} \hookrightarrow Sp_{2n}$.  Write $Sp := \colim_n Sp_{2n}$ with respect to these embeddings.  The quotients $Sp_{2n}/Sp_{2n-2}$ exist and are isomorphic to $Q_{4n-1}$.  As a consequence, for any integer $n \geq 1$, there are $\aone$-fiber sequences of the form
\[
Sp_{2n-2} \longrightarrow Sp_{2n} \longrightarrow {\mathbb A}^{2n} \setminus 0.
\]
For any integer $j \geq 0$, these $\aone$-fiber sequences induce long exact sequences in $\aone$-homotopy sheaves of the form
\[
\cdots \longrightarrow \bpi_{i,j}^{\aone}(Sp_{2n-2}) \longrightarrow \bpi_{i,j}^{\aone}(Sp_{2n}) \longrightarrow \bpi_{i,j}^{\aone}({\mathbb A}^{2n} \setminus 0) \longrightarrow \bpi_{i-1,j}^{\aone}(Sp_{2n-2}) \longrightarrow \cdots.
\]
Putting these sequences together, we get an exact couple
\[
\xymatrix{
\bigoplus_{n \geq 1, i \geq0} \bpi_{i,j}^{\aone} (Sp_{2n}) \ar[rr]& & \bigoplus_{n \geq 1, i\geq0}\bpi_{i,j}^{\aone}(Sp_{2n}) \ar[dl]\\ & \bigoplus_{n \geq 1, i\geq0} \bpi_{i,j}^{\aone}({\mathbb A}^{2n} \setminus 0) \ar[ul].&
}
\]
Once again, $\colim_n \bpi_{i,j}^{\aone}(Sp_{2n}) = \bpi_{i,j}^{\aone}(Sp)$.  Using this fact, we may analyze the spectral sequence associated with the above exact couple.  We deduce the following result, whose proof is formally identical to that of Proposition \ref{prop:linearss}.

\begin{prop}
\label{prop:symplecticss}
Assume $j \geq 0$ is an integer.
\begin{enumerate}[noitemsep,topsep=1pt]
\item There is a spectral sequence with
\[
E^1_{p,q}(j) := \bpi_{p+q,j}^{\aone}({\mathbb A}^{2p} \setminus 0)
\]
and converging to $\bpi_{p+q,j}^{\aone}(Sp) = \mathbf{GW}^{2-j}_{p+q-j+1}$.
\item For any $j \geq 0$, $E_{p,q}^1(j) = 0$ for i) $p < 0 $, and ii) $q \leq p-2$.
\end{enumerate}
\end{prop}

\begin{proof}
Convergence is established exactly as in Proposition \ref{prop:linearss}.  The identification of the $\aone$-homotopy sheaves of the stable symplectic group with the higher Grothendieck-Witt sheaves follows from the Schlichting-Tripathi representability theorem for Hermitian K-theory \cite[Theorem 8.2]{SchlichtingTripathi}.  The second statement is an immediate consequence of \cite[Corollary 6.39]{MField}.
\end{proof}

The above vanishing statement, together with the identification $\bpi_1^{\aone}({\mathbb A}^2 \setminus 0) \cong \K^{MW}_2$ immediately implies the following result.

\begin{cor}
There are low-dimensional isomorphisms $\mathbf{GW}^2_1 \cong 0$ and $\mathbf{GW}^2_2 \cong \K^{MW}_2$.
\end{cor}

\begin{rem}
This sheaf $\mathbf{GW}^2_2$ is by definition the second symplectic K-theory sheaf $\K^{Sp}_2$.  That this sheaf is identified with $\K^{MW}_2$ is essentially a result of Suslin \cite{SuslinTorsion}; see \cite[Theorem 4.1.2]{AsokFaselKO} for more details.
\end{rem}

\subsection{The anti-symmetric spectral sequence}
\label{ss:antisymmetricss}
Consider the inclusion map $Sp_{2n} \hookrightarrow GL_{2n}$.  The quotients $GL_{2n}/Sp_{2n}$ exist as smooth schemes and we set $X_n = GL_{2n}/Sp_{2n}$.  There are evident stabilization maps $X_n \to X_{n+1}$ induced by the stabilization maps $GL_{2n-2} \to GL_{2n}$ and $Sp_{2n-2} \to Sp_{2n}$.  We set $GL/Sp := \colim_n X_n$.

For any integer $n \geq 1$, the stabilization maps fit into $\aone$-fiber sequences of the form
\[
X_n \longrightarrow X_{n+1} \longrightarrow {\mathbb A}^{2n+1} \setminus 0;
\]
(see \cite[Proposition 4.2.2]{AsokFaselpi3a3minus0} and note the same argument works when one replaces the special linear group by the general linear group).  These $\aone$-fiber sequences yield, for any integer $j \geq 0$, long exact sequences in $\aone$-homotopy sheaves of the form
\[
\cdots \longrightarrow \bpi_{i,j}^{\aone}(X_n) \longrightarrow \bpi_{i,j}^{\aone}(X_{n+1}) \longrightarrow \bpi_{i,j}^{\aone}({\mathbb A}^{2n+1} \setminus 0) \longrightarrow \bpi_{i-1,j}^{\aone}(X_n) \longrightarrow \cdots.
\]
One can put these sequences together to obtain an exact couple, and regarding the associated spectral sequence, one has the following result.

\begin{prop}
\label{prop:antisymmetricss}
Assume $j \geq 0$ is an integer.
\begin{enumerate}[noitemsep,topsep=1pt]
\item There is a spectral sequence with
\[
E^1_{p,q}(j) := \bpi_{p+q,j}^{\aone}({\mathbb A}^{2p-1} \setminus 0)
\]
and converging to $\bpi_{p+q,j}^{\aone}(GL/Sp) = \mathbf{GW}^{3-j}_{p+q-j+1}$.
\item For any $j \geq 0$, $E_{p,q}^1(j) = 0$ for i) $q \leq p-3$, ii) $p < 0$, and iii) $p = 1$ and $q > 0$.
\end{enumerate}
\end{prop}

\begin{proof}
Again, the proof is essentially identical to Proposition \ref{prop:linearss}, though this time the identification of the higher $\aone$-homotopy sheaves of $GL/Sp$ with higher Grothendieck-Witt sheaves follows from \cite[Theorem 8.3]{SchlichtingTripathi}.
\end{proof}

The above result, together with $\bpi_0^{\aone}(\gm{}) \cong \gm{}$, and Morel's computation $\bpi_2^{\aone}({\mathbb A}^3 \setminus 0)$ immediately implies the following result.

\begin{cor}
There are low-dimensional isomorphisms $\mathbf{GW}^3_1 \cong \gm{}$, $\mathbf{GW}^3_2 \cong 0$ and $\mathbf{GW}^3_3 \cong \K^{MW}_3$.
\end{cor}

\begin{rem}
The identification $\mathbf{GW}^3_3 \cong \K^{MW}_3$ is that analyzed in \cite[Theorem 4.3.1]{AsokFaselKO}.
\end{rem}

\section{Some results on odd primary torsion in $\aone$-homotopy groups}
\label{s:oddtorsion}
Serre showed \cite{Serre} that, if $p$ is a prime, then the first $p$-torsion in the higher homotopy groups of $S^3$ appears in $\pi_{2p}(S^3)$.  The classical proof of this result relies on an analysis of the Serre spectral sequence for the fibration $BS^1 \to S^3 \langle 3 \rangle \to S^3$ and Serre class theory.  The goal of this section is to begin to lift this computation to unstable $\aone$-homotopy theory, in which, at the time of writing, neither of the two tools just mentioned are available.

\subsection{Odd primary torsion and the topological symplectic spectral sequence}
\label{ss:topologicalsymplecticss}
If $Sp(2n)$ is the compact real form of the symplectic group, there are topological fiber sequences of the form
\[
Sp(2n-2) \longrightarrow Sp(2n) \longrightarrow S^{4n-1}.
\]
Putting the long exact sequences in homotopy groups associated with these fibrations together yields an exact couple and an associated spectral sequence with
\[
E^1_{p,q} = \pi_{p+q}(S^{4p-1}) \Longrightarrow \pi_{*}(Sp(\infty)).
\]
By analogy with the ``orthogonal spectral sequence" mentioned before, we will refer to this spectral sequence as the topological symplectic spectral sequence.  The differentials appearing in this spectral sequence will bear a superscript ``$top$" to distinguish them from those appearing in the spectral sequence constructed in Subsection \ref{ss:symplecticss}.  The homotopy groups of $Sp(\infty)$ are known by Bott periodicity and we now use this to interpret the $p$-torsion in $\pi_{2p}(S^3)$ in terms of differentials in this spectral sequence.

\begin{prop}
\label{prop:topsss}
Suppose $\ell$ is an odd prime.  The generator of the $\ell$-torsion of $\pi_{2\ell}(S^3)$ is the image of an element of $\pi_{2\ell+1}(S^{2\ell+1})$ under the differential $d_{(\ell-1)/2}^{top}$ in the topological symplectic spectral sequence.
\end{prop}

\begin{proof}
First, note that if $F \to E \to B$ is a Serre fibration, then there is a corresponding long exact sequence in homotopy groups mod $n$ for any integer $n$ \cite[Proposition 1.6]{Neisendorfer}.  In particular, the exact couple associated with the long exact sequence in homotopy of the fiber sequences $Sp(2n-2) \to Sp(2n) \to S^{4n - 1}$ yields a $\mod n$ topological symplectic spectral sequence:
\[
E^1_{p,q} = \pi_{p+q}(S^{4p-1};\Z/\ell\Z) \Longrightarrow \pi_{*}(Sp(\infty);\Z/\ell\Z).
\]
Classical results allow us to completely describe the $E^1$-page in a range and Bott periodicity allows us to understand precisely to what the spectral sequence converges.

In more detail, $\pi_{4p-1}(S^{4p-1},\Z/\ell\Z) = \Z/\ell\Z$.  On the other hand, if $n$ is an odd number and $k > 0$, the first non-trivial $\ell$-torsion in $\pi_{n+k}(S^n)$ appears in degree $k = 2 \ell - 3$ \cite[Proposition 11 on p. 285]{Serre}.  In particular, the $E^1$-page of the mod $\ell$ symplectic spectral sequence takes a rather simple form.  On the other hand, since $\ell$ is an odd prime, Bott periodicity \cite[Corollary to Theorem II]{BottPeriodicity} implies that the $\ell$-completion of $\pi_i(Sp({\infty}))$ is non-trivial if and only if $i$ is congruent to $3$ mod $4$.  In particular, $\pi_{2\ell}(Sp({\infty}),\Z/\ell) = 0$ since $2\ell$ is even.  Combining these observations, we deduce that the mod $\ell$ topological symplectic spectral sequence is particularly degenerate: the only non-trivial differential that lands on $\pi_{2\ell}(S^3)$ is $d^{top}_{(\ell-1)/2}$, and this differential is necessarily an isomorphism since $\pi_{2p}(Sp({\infty}),\Z/p)$ is trivial.
\end{proof}

\begin{lem}
\label{lem:topologicalsurjectivity}
The map
\[
\operatorname{ker}(\pi_{11}(S^{11}) \to \pi_{10}(S^{7})) \stackrel{d_2^{top}}{\longrightarrow}  \pi_{10}(S^3)/\operatorname{im}(\pi_{11}(S^7))
\]
is surjective.  Moreover, $\pi_{10}(S^3)$ is generated by $d_2^{top}(24 \iota)$, where $\iota$ is a generator of $\pi_{11}(S^{11})$.
\end{lem}

\begin{proof}
The proof of this fact amounts to recalling some classical computations of homotopy groups of spheres.  Let $\nu^{top}: S^7 \to S^4$ be the usual Hopf map.  We abuse terminology and also write $\nu^{top}: S^{n+3} \to S^n$ for any $n \geq 4$ for the maps obtained by iterated suspension.  It is well known that for $n \geq 5$, $\nu^{top}$ generates $\pi_{n+3}(S^n) \cong \Z/24$.  Likewise, following Toda, we write $(\nu')^{top}: S^6 \to S^3$ for the class obtained from the classifying map of $S^7 = Sp(4)/Sp(2) \to BSp(2)$ by adjunction.  Serre observed in \cite[p. 285 Remarques (2)]{Serre} that the generator of $\pi_9(S^3)$ is precisely the composite of a $2$-fold suspension of the Hopf map $\nu^{top}$, generating $\pi_9(S^6)$, and $(\nu')^{top}: S^6 \to S^3$.  By \cite[Proposition 5.15]{Toda} and \cite[p. 285 Corollaire]{Serre} one knows that $\pi_{10}(S^3) \cong \Z/15$ and by \cite[Proposition 5.8]{Toda} $\pi_{11}(S^7) = 0$.

Granted the computations mentioned in the previous paragraph, to establish the result, it suffices to show that $d^2_{top}$ is surjective after reduction modulo $3$ and reduction modulo $5$.  The statement after reduction modulo $5$ follows immediately from Proposition \ref{prop:topsss}.  After reduction modulo $3$, observe that $d_1: \pi_{10}(S^7,\Z/3) \to \pi_9(S^3,\Z/3)$ is a map $\Z/3 \to \Z/3$.    Since the $d_1$ differential at this stage is by construction of the symplectic spectral sequence induced by composition with the connecting map, it follows that $d_1$ is an isomorphism after reduction modulo $3$.  Therefore, it follows that the induced map $\pi_{11}(S^{11}) \to \pi_{10}(S^3)$ is surjective after reduction modulo $3$ as well.
\end{proof}

\begin{rem}
The group $\pi_{14}(S^3) \cong \Z/84 \times \Z/2^{\times 2}$ is not cyclic; see, e.g., \cite[Theorem 7.4]{Toda} for the computation of the $2$-component.  Therefore the differential $d_3^{top}$ cannot be surjective integrally.
\end{rem}

\subsection{Lifting some odd primary torsion from topology to $\aone$-homotopy}
\label{ss:oddtorsionina2minus0}
Complex realization does not, in general, preserve fiber sequences.  Nevertheless, since a) $Sp_{2n}^{an}$ is homotopy equivalent to the compact symplectic group $Sp(2n)$ and b) the $Q_{4n-1}^{an}$ is homotopy equivalent to $S^{4n-1}$, the complex realization of the $\aone$-fiber sequence $Sp_{2n-2} \to Sp_{2n} \to Q_{4n-1}$ is the topological fiber sequence $Sp(2n-2) \to Sp(2n) \to S^{4n-1}$.  Therefore, complex realization determines a morphism from the exact couple giving rise to the symplectic spectral sequence we considered in Subsection \ref{ss:symplecticss} evaluated on complex points to the exact couple that gives rise to the topological symplectic spectral sequence described in Subsection \ref{ss:topologicalsymplecticss}.  As a consequence, there is an induced morphism of spectral sequences.  By the discussion of the introduction, the next result implies Theorem \ref{thmintro:main} from the introduction.

\begin{thm}
\label{thm:oddtorsion}
If $p = 3,5$, then the homomorphism $\bpi_{p-1,p+1}^{\aone}({\mathbb A}^2 \setminus 0)(\cplx) \to \pi_{2p}(S^3)$ is surjective.
\end{thm}

\begin{proof}
For $p = 3$, this is worked out in \cite[Theorem 7.5]{AsokFaselThreefolds} (take $n = 1$ and observe that $Sp_2(\cplx) = SL_2(\cplx)$ is homotopy equivalent to $S^3$).  We treat the case $p =5$.  We study the map from the $E_2$-page of the symplectic spectral sequence to the $E_2$-page of the topological symplectic spectral sequence whose existence is guaranteed by the discussion just prior to the theorem statement.  In particular, there is a commutative diagram of the form
\[
\xymatrix{
\operatorname{ker}(\bpi_{5,6}^{\aone}({\mathbb A}^{6} \setminus 0)(\cplx) \to \bpi_{4,6}^{\aone}({\mathbb A}^4 \setminus 0)(\cplx)) \ar[r]^-{d_2} \ar[d] & \bpi_{4,6}^{\aone}({\mathbb A}^2 \setminus 0)(\cplx)/\operatorname{im}(\bpi_{5,6}^{\aone}({\mathbb A}^4 \setminus 0)(\cplx)) \ar[d] \\
\operatorname{ker}(\pi_{11}(S^{11}) \to \pi_{10}(S^{7})) \ar[r]^{d_2^{top}} & \pi_{10}(S^3)/\operatorname{im}(\pi_{11}(S^7)).
}
\]
We now analyze this diagram.

Observe that $d_2^{top}$ is surjective (with a precise generator identified) by Lemma \ref{lem:topologicalsurjectivity}. It follows from \cite[Corollary 6.39]{MField} that for any integer $n \geq 1$, $\bpi_{n,n+1}^{\aone}({\mathbb A}^{n+1} \setminus 0)(\cplx) \cong \K^{MW}_0(\cplx) \cong \Z$ and a generator of the group $\bpi_{n,n+1}^{\aone}({\mathbb A}^{n+1} \setminus 0)(\cplx)$ is sent to a generator of $\pi_{2n+1}(S^{2n+1})$ under complex realization, so the left vertical map is an injection.  We will show that it is split injective.

Now, recall that the connecting homomorphism in the long exact sequence in homotopy sheaves associated with the $\aone$-fiber sequence $SL_4/Sp_4 \to SL_6/Sp_6 \to {\mathbb A}^5 \setminus 0$ determines a morphism of sheaves $\bpi_{4,5}^{\aone}({\mathbb A}^5 \setminus 0) \to \bpi_{3,5}^{\aone}({\mathbb A}^3 \setminus 0)$; equivalently, this is the $d_1$-differential in the anti-symmetric spectral sequence of Section \ref{ss:antisymmetricss}.  Now, by construction, the connecting homomorphism is induced by a map $\Omega ({\mathbb A}^5 \setminus 0) \longrightarrow {\mathbb A}^3 \setminus 0$.  The unit of the loop-suspension adjunction determines a map $\Sigma^3 \gm{\sma 5} \to \Omega ({\mathbb A}^5 \setminus 0)$; this morphism is an isomorphism on $\pi_3^{\aone}$ so the map on homotopy sheaves we consider is induced by applying $\bpi_{3,5}^{\aone}(-)$ to a morphism $\Sigma^3 \gm{\sma 5} \to {\mathbb A}^3 \setminus 0$.  Taking the $\pone$-suspension of this morphism, we obtain a commutative diagram of the form:
\[
\xymatrix{
\bpi_{3,5}^{\aone}(\Sigma^3 \gm{\sma 5})(\cplx) \ar[r] \ar[d] & \bpi_{3,5}^{\aone}({\mathbb A}^3 \setminus 0)(\cplx)\ar[d] \\
\bpi_{4,6}^{\aone}(\Sigma^4 \gm{\sma 6})(\cplx) \ar[r] & \bpi_{4,6}^{\aone}({\mathbb A}^{4} \setminus 0)(\cplx))
},
\]
where the top morphism coincides with the $d_1$-differential in the anti-symmetric spectral sequence under the isomorphism $\bpi_{3,5}^{\aone}(\Sigma^3 \gm{\sma 5})(\cplx) \cong \bpi_{4,5}^{\aone}({\mathbb A}^5 \setminus 0)$.


The left hand vertical map is an isomorphism by appeal to \cite[Corollary 6.39]{MField}.  On the other hand, $\bpi_{3,5}^{\aone}({\mathbb A}^3 \setminus 0)(\cplx) \cong \Z/24$, generated by a class $\delta$ by \cite[Proposition 5.2.1]{AsokFaselpi3a3minus0}.  Moreover, by analyzing the proof of \cite[Proposition 5.1]{AsokFaselpi3a3minus0}, one observes that $\delta$ is the image of a generator $\iota$ of $\bpi_{3,5}^{\aone}(\Sigma^3\gm{\sma 5})(\cplx)$ as a $GW(\cplx) = \Z$-module under the top horizontal map; in other words, $\iota$ is sent to a generator of $\Z/24$.  Now $\delta$ is stably non-trivial: indeed, complex realization sends $\delta$ to a generator of $\pi_8(S^5) = \Z/24$ and thus $\delta$ differs from $\Sigma\nu$, which is classically known to be stably non-trivial, by a unit in $\Z/24$.  It follows that $\Sigma_{\pone}\delta$ is non-zero in $\bpi_{4,6}^{\aone}({\mathbb A}^{4} \setminus 0)(\cplx))$ under the right vertical map and commutativity shows that $\Sigma_{\pone}\delta$ is the image of $\Sigma_{\pone}\iota$.

We also observed in \cite[Corollary 5.3.1]{AsokFaselpi3a3minus0} that $\delta$ is sent to the suspension of $\nu^{top}$ under complex realization, and by compatibility of complex realization and suspension, it follows that complex realization sends the $\pone$-suspension of $\delta$ to a threefold suspension of $\nu^{top}$.  This provides the splitting mentioned in the previous paragraph.  Combining these two observations, we obtain the splitting mentioned two paragraphs above.

The generator $24 \Sigma_{\pone}\iota$ of $24\Z \subset \bpi_{4,6}^{\aone}(\Sigma^4 \gm{\sma 6})(\cplx)$ is sent to a generator of $24\Z \subset \pi_{10}(S^{10})$ under complex realization.  Under the isomorphism $\pi_{10}(S^{10}) \to \pi_{11}(S^{11})$, the latter is sent by  $d_2^{top}$ in the topological symplectic spectral sequence to a generator of $\bpi_{10}(S^3)$, it follows that $d_2(24 \Sigma_{\pone}\iota)$ lifts this generator in $\bpi_{4,6}^{\aone}({\mathbb A}^2 \setminus 0)$.
\end{proof}

\subsection{Complements}
It is possible to establish a result like Theorem \ref{thm:oddtorsion} using the anti-symmetric spectral sequence.  Since the proof is essentially identical to proof of Theorem \ref{thm:oddtorsion} it will only be sketched.

\begin{thm}
The homomorphism $\bpi_{5,7}^{\aone}({\mathbb A}^3 \setminus 0) \to \bpi_{12}(S^5) \cong \Z/30$ is surjective.
\end{thm}

\begin{proof}
If $\Gamma_n := U(2n)/Sp(2n)$, then there are fiber sequences of the form $\Gamma_n \to \Gamma_{n+1} \to S^{2n+1}$.  The long exact sequences in homotopy fit together to yield an exact couple that is the topological analog of the anti-symmetric spectral sequence considered in Subsection \ref{ss:antisymmetricss}.  This spectral sequence converges to the homotopy groups of $U/Sp$, which are known by Bott periodicity.  Observe that $\pi_{2i}(U/Sp) \cong \pi_{2i-2}(Sp)$ and so vanishes by explicit computation.  Again, complex realization yields a morphism from the anti-symmetric spectral sequence to its topological counterpart.

Consider then the commutative diagram:
\[
\xymatrix{
\operatorname{ker}(\bpi_{6,7}^{\aone}({\mathbb A}^{7} \setminus 0)(\cplx) \to \bpi_{5,7}^{\aone}({\mathbb A}^5 \setminus 0)(\cplx)) \ar[r]^-{d_2} \ar[d] & \bpi_{5,7}^{\aone}({\mathbb A}^3 \setminus 0)(\cplx)/\operatorname{im}(\bpi_{6,7}^{\aone}({\mathbb A}^5 \setminus 0)(\cplx)) \ar[d] \\
\operatorname{ker}(\pi_{13}(S^{13}) \to \pi_{12}(S^{9})) \ar[r]^{d_2^{top}} & \pi_{12}(S^5)/\operatorname{im}(\pi_{13}(S^9)).
}
\]
Again, $\pi_{13}(S^9) = 0$ and one observes that $d_2^{top}$ is surjective in a fashion identical to Lemma \ref{lem:topologicalsurjectivity}.  The remainder of the analysis is analogous to the end of the proof of Theorem \ref{thm:oddtorsion}.
\end{proof}

\section{Building explicit representatives}
\label{s:explicit}
Given the existence of at least $15$ non-isomorphic rank $2$ algebraic vector bundles on $Q_{11}$ it would be interesting to construct explicit representatives of these bundles.  It follows from the results of \cite{AsokFaselSpheres}, which we review below, that all such bundles are stably trivial.  Corollary \ref{cor:nontriviality} demonstrates that non-trivial rank $2$ algebraic vector bundles on $Q_{11}$ whose associated topological bundles are non-trivial, remain algebraically non-trivial after forming the direct sum with trivial bundles of rank $\leq 3$.  

\subsection{Stable triviality results}
The inclusion of $M \in SL_n(R)$ in $SL_{n+1}(R)$ as block diagonal matrices of the form $diag(M,1)$ gives a morphism of spaces $BSL_n \to BSL_{n+1}$.  If $X$ is a smooth affine scheme, then the induced map
\[
[X,BSL_n]_{\aone} \to [X,BSL_{n+1}]_{\aone}
\]
corresponds to the operation of adding a trivial rank $1$ bundle.  By means of the identifications mentioned in the introduction, when $X = Q_{2i-1}$, the above function corresponds to a morphism
\[
\Phi_{1,n}: \bpi_{i-1,i}^{\aone}(BSL_n) \to \bpi_{i-1,i}^{\aone}(BSL_{n+1}).
\]
Write $\Phi_{m,n}$ for the composite morphism $\Phi_{1,n+m-1} \circ \cdots \circ \Phi_{1,n+1}\circ \Phi_{1,n}$.

To answer the question of whether a bundle on $Q_{11}$ becomes trivial after successively adding trivial bundles of rank $1$ amounts to studying whether a class $\xi \in \bpi_{5,6}^{\aone}(BSL_2)(\cplx)$ is sent to $0$ under $\Phi_{m,2}$.  The next result shows that $\Phi_{m,2}$ is the zero map for $m \geq 4$.

\begin{lem}[{\cite[Corollary 4.7]{AsokFaselSpheres}}]
\label{lem:stabletriviality}
If $n \geq 1$, and $m \geq n$, $\mathscr{V}_m(Q_{2n-1}) = \ast$.
\end{lem}

\subsection{Adding trivial summands of small rank}
The topological analog of the sequence of homomorphisms considered in the previous section can be analyzed using classical results.  Precisely, we have the following result.

\begin{prop}
\label{prop:non-triviality}
The homomorphisms $\pi_{11}(BSU(2)) \to \pi_{11}(BSU(m))$ are injective for $3 \leq m \leq 5$.
\end{prop}

\begin{proof}
We begin by recalling various computations of homotopy groups of special unitary groups.
One knows that $\pi_{11}(BSU(2)) = \pi_{10}(SU(2)) = \pi_{10}(S^3) \cong \Z/15$ \cite[p. 186]{Toda}, $\pi_{11}(BSU(3)) \cong \Z/30$ \cite[Theorem 6.1]{MimuraToda}, $\pi_{11}(BSU(4)) \cong \Z/2 \oplus \Z/5!$  \cite[Lemma I.6]{Kervaire}, and $\pi_{11}(BSU(5)) \cong \Z/5!$ \cite[Theorem 5]{Bott}.  For injectivity when $m = 3$, consider the portion of the long exact sequence in homotopy groups associated with $S^5 \to BSU(2) \to BSU(3)$
\[
\pi_{11}(S^5) \longrightarrow \pi_{11}(BSU(2)) \longrightarrow \pi_{11}(BSU(3)).
\]
We know that $\pi_{11}(S^5) \cong \Z/2$ \cite[p. 186]{Toda}, so the left hand map is zero.  Therefore, the map $\pi_{11}(BSU(2)) \to \pi_{11}(BSU(3))$ must be injective.

Next, consider the long exact sequence in homotopy groups associated with $S^7 \to BSU(3) \to BSU(4)$.  In that case, we have
\[
\pi_{11}(S^7) \longrightarrow \pi_{11}(BSU(3)) \longrightarrow \pi_{11}(BSU(4)) \longrightarrow \pi_{10}(S^7).
\]
In this case, $\pi_{11}(S^7) = 0$ \cite[p. 186]{Toda}, so the map $\pi_{11}(BSU(3)) \to \pi_{11}(BSU(4))$ is injective.  Combining with the conclusion of the previous paragraph, injectivity for $m = 4$ is settled.

For the case $m = 5$, first observe that since $\pi_{11}(BSU(2)) \cong \Z/15$ its image in $\pi_{11}(BSU(4))$ is necessarily contained in the summand isomorphic to $\Z/5!$.  Kervaire's computation of $\pi_{11}(BSU(4))$ mentioned above proceeds by analysis of the long exact sequence in homotopy attached to the fiber sequence $S^9 = SU(5)/SU(4) \to BSU(4) \to BSU(5)$. Indeed, this exact sequence takes the form:
\[
\pi_{11}(S^9) \longrightarrow \pi_{11}(BSU(4)) \longrightarrow \pi_{11}(BSU(5)) \longrightarrow 0,
\]
where the left hand group is $\Z/2$ and so the map $\pi_{11}(BSU(4)) \to \pi_{11}(BSU(5))$ is the projection onto the summand isomorphic to $\Z/5!$.  The result follows.
\end{proof}

The next result is a straightforward consequence of Lemma \ref{lem:stabletriviality}, Proposition \ref{prop:non-triviality}, Theorem \ref{thm:oddtorsion} and the functoriality of complex realization.

\begin{cor}
\label{cor:nontriviality}
If $\xi$ is an element of $\bpi_{5,6}^{\aone}(BSL_2)(\cplx)$ that does not lie in the kernel of the (surjective) complex realization map $\bpi_{5,6}^{\aone}(BSL_2)(\cplx) \to \pi_{11}(BSU(2))$, then the image of $\xi$ under the homomorphism
\[
\Phi_{m,2}: \bpi_{5,6}^{\aone}(BSL_2)(\cplx) \longrightarrow \bpi_{5,6}^{\aone}(BSL_m)(\cplx)
\]
is non-trivial for $3 \leq m \leq 5$.
\end{cor}

\begin{footnotesize}
\bibliographystyle{alpha}
\bibliography{algebraizability}
\end{footnotesize}
\Addresses
\end{document}